\documentclass[12pt, reqno]{amsart}
\usepackage{amsmath, amsthm, amscd, amsfonts, amssymb, graphicx, color, mathrsfs}
\usepackage[bookmarksnumbered, colorlinks, plainpages]{hyperref}
\usepackage{slashed}
\usepackage{mathtools}
\usepackage{enumitem}

\usepackage{lineno,hyperref}

 \newtheorem{thm}{Theorem}[section]
 \newtheorem{cor}[thm]{Corollary}
 
 \newtheorem{prop}[thm]{Proposition}
 \theoremstyle{definition}
 \newtheorem{defn}[thm]{Definition}
 \theoremstyle{remark}
 \newtheorem{rem}[thm]{Remark}
 \newtheorem{ex}[thm]{Example}
 \numberwithin{equation}{section}
\newcommand{\half}{\frac{1}{2}}

\def\SU2{{{\rm SU(2)}}}
\newcommand{\er}{\mathbb{R}}

\newcommand{\tn}{\mathbb{T}^n}

\newcommand{\ern}{{\mathbb{R}}^n}

\newcommand{\bi}{\begin{itemize}}
\newcommand{\ei}{\end{itemize}}
\newcommand{\be}{\begin{enumerate}}
\newcommand{\ee}{\end{enumerate}}
\newcommand{\beq}{\begin{equation}}
\newcommand{\eq}{\end{equation}}

\newcommand{\lapm}{\Delta_M}

\def\SO3{{{\rm SO(3)}}}

\def\p#1{{\left({#1}\right)}}

\DeclareMathOperator{\Tr}{Tr}

\def\Rn{{{\mathbb R}^n}}
\def\Rm{{{\mathbb R}^m}}
\def\Tn{{{\mathbb T}^n}}
\def\Zn{{{\mathbb Z}^n}}
\def\Zm{{{\mathbb Z}^m}}
\def\T{{{\mathbb T}^1}}
\def\TT{{{\mathbb T}^2}}

\def\C{{{\mathbb C}}}
\def\SU2{{{\rm SU(2)}}}
\def\lapsu2{{{\mathcal L}_\SU2}}

\begin{document}

\title[Schatten-von Neumann classes of integral operators]{Schatten-von Neumann classes of integral operators}

\author[J. Delgado]{Julio Delgado}

\address{Julio Delgado: 
  \endgraf
 Universidad del Valle \\
Departamento de Matem\'aticas\\
Colombia}

\email{delgado.julio@correounivalle.edu.co}

\author[M. Ruzhansky]{Michael Ruzhansky}
\address{
  Michael Ruzhansky:
  \endgraf
  Department of Mathematics: Analysis, Logic and Discrete Mathematics
  \\
  Ghent University, Belgium
  \endgraf
  and
  \endgraf
  School of Mathematical Sciences \\
    Queen Mary University of London
  \\
  United Kingdom}
  \email{michael.ruzhansky@ugent.be}
  
\date{\today}

\begin{abstract}
In this work we establish  sharp kernel conditions ensuring that the  corresponding  integral operators belong to Schatten-von Neumann classes. The conditions are given in terms of the spectral properties of operators acting on the kernel. As applications we establish several criteria in terms of different types of differential operators and their spectral asymptotics in different settings: compact manifolds, operators on lattices, domains in ${\mathbb R}^n$ of finite measure, and conditions for operators on ${\mathbb R}^n$ given in terms of anharmonic oscillators. We also give examples in the settings of compact sub-Riemannian manifolds, contact manifolds, strictly pseudo-convex CR manifolds, and (sub-)Laplacians on compact Lie groups.
\end{abstract}


\maketitle
\tableofcontents

\section{Introduction}
Let $(\Omega_j,\mathcal{M}_j,\mu_j)$ ($j=1, 2$) be measure spaces respectively endowed with  a $\sigma$-finite measure $\mu_j$ on a $\sigma$-algebra $\mathcal{M}_j$ of subsets of $\Omega_j$. We denote $L^2(\mu_j):=L^2(\Omega_j,\mu_j)$ the complex Hilbert space of square integrable functions on $\Omega_j$. In this paper we give 
{\em sharp sufficient conditions} on integral kernels $K=K(x,y)$ in order to ensure that 
the corresponding integral operators 
\begin{equation}\label{EQ:int1}
Tf(x)=\int_{\Omega_1} K(x,y) f(y) d\mu_1(y)
\end{equation} 
from $L^2(\mu_1)$ into $L^2(\mu_2)$ belong to
different Schatten-von Neumann classes and in particular to the trace class when $(\Omega_1,\mathcal{M}_1,\mu_1)=(\Omega_2,\mathcal{M}_2,\mu_2)$.  When $\Omega=\Omega_1=\Omega_2$ possesses a Borel topological structure a trace formula  
in terms of the diagonal of the kernel can be deduced. If additionally $\Omega$ has a smooth manifold structure some sharp sufficient conditions on integral kernels $K(x,y)$ for Schatten-von Neumann classes can be formulated in terms of the regularity of a certain order in either $x$ or $y$, or both, and in terms of decay conditions at infinity.

To briefly explain the approach of this paper, we can summarise it as follows:
\begin{center}
{\em If we know spectral properties of an operator $E$ and we known how it acts on the integral kernel of an integral operator $T$, we can draw conclusions about the spectral properties of $T$.}
\end{center}
More specifically, let $(E_2)_x$ and $(E_1)_y$ be operators acting on $x$ and $y$ variables, respectively, and suppose we know that
\begin{equation}\label{EQ:int2}
(E_2)_x(E_1)_yK\in L^2(\mu_2\otimes\mu_1),
\end{equation} 
or, more generally, belongs to mixed Lebesgue spaces.
In this paper we give {\em spectral} conditions on $E_2$ and $E_1$ ensuring that the integral  operator $T$ in \eqref{EQ:int1} belongs to the Schatten-von Neumann class $S_r(L^2(\mu_1), L^2(\mu_2))$, $0<r<\infty$. Such spectral conditions on $E_2$ and $E_1$ will be given:
\begin{itemize}
\item in terms of membership of their inverses in Schatten-von Neumann classes;
\item or in terms of their spectral asymptotics,
\end{itemize}
 i.e. conditions that can be verified in practice.
As an application we present several tests in terms of different types of operators in different settings. While the knowledge of the spectral asymptotics of an operator $E$ implies also the Schatten-von Neumann properties for its inverse, the advantage of knowing the spectral asymptotics will be in 
a possibility to obtain refined properties of the decay of the singular numbers of the integral operator $T$.

 The problem of finding such criteria on different kinds of domains is classical and has been much studied, e.g. the paper \cite{bs:sing} by Birman and Solomyak is a good introduction to the subject. In particular, it is well known that the smoothness of the kernel is related to the behaviour of the singular numbers. In a recent paper \cite{dr:suffkernel} the authors have established sufficient sharp kernel conditions for operators on $L^2(M)$ for a compact manifold $M$. 
 
In order to obtain criteria for general Schatten-von Neumann classes we will use the well-known method of factorisation introduced by Gohberg and Krein (cf. \cite[Chap. 3, Section 10]{gokr}).  Some applications of this method have been developed by O'Brien in \cite{OBrien:test} and in \cite{OBrien:test2} for applications to trace formulae of Schr{\"o}dinger operators. The factorisation method has also been applied by the authors in \cite{dr:suffkernel} in order to obtain Sobolev regularity conditions for kernels on closed manifolds. The main result of \cite{dr:suffkernel} has been recently applied in \cite{sm:dlay} for the study of the distribution of eigenvalues for the double layer potential on flat curves and its relation with isoperimetrical inequalities.

There is an extensive literature on Schatten-von Neumann properties for specific integral operators. A characterisation of the membership in Schatten-von Neumann classes for Hardy operators has been established by Nowak in \cite{kn:sch}. Lower bounds for the Schatten-von Neumann norms for Volterra integral operators have been obtained in  \cite{edstep:sch}, \cite{step:sch}. Schatten-von Neumann classes of pseudo-differential operators in the setting of the Weyl-H\"or\-man\-der calculus have been considered in 
\cite{Toft:Schatten-AGAG-2006}, \cite{Toft:Schatten-modulation-2008}, \cite{Buzano-Toft:Schatten-Weyl-JFA-2010} and for anharmonic oscillators in  \cite{anh:cdr}. Schatten-von Neumann classes on compact Lie groups and $s$-nuclear operators on $L^p$ spaces from the point of view of symbols have been respectively studied by the authors in \cite{dr13:schatten} and \cite{dr13a:nuclp}. 

In the subsequent part of the 
present paper we establish a
characterisation of Schatten-von Neumann classes for operators from $L^2(\mu_1)$ into $L^2(\mu_2)$ for general integral operators on general measure spaces.

The role of the smoothness of the kernel  in the investigation of Schatten-von Neumann properties can be illustrated with the following example. 
 In his classical book \cite[Prop 3.5, page 174]{sugiura:book} Mitsuo Sugiura gives a trace class criterion for integral operators on $L^2(\T)$ with $C^2$-kernels. More precisely, the theorem asserts that every kernel in $C^2(\TT)$ begets a trace class operator on $L^2(\T)$: if $K(\theta,\phi)$ is a $C^2$-function on $\TT$, then the integral operator $L$ on 
 $L^2(\T)$ defined by
\begin{equation}\label{EQ:Sugiura-trace-op} 
 (Lf)(\theta)=\int_0^{2\pi}K(\theta,\phi)f(\phi)d\phi,
\end{equation}
is trace class and has the trace
\begin{equation}\label{EQ:Sugiura-trace}
\Tr(L)=\int_0^{2\pi}K(\theta,\theta)d\theta.
\end{equation}
The proof of this result relies on the connection between the absolute convergence of Fourier coefficients of the kernel and the trace class property (traceability) of the corresponding operator. This result has been extended by the authors under sharp smoothness assumptions to the setting of compact manifolds  in \cite{dr:suffkernel}.   The authors have weakened the known
assumptions on the kernel for the operator to be trace class and for the trace formula
\eqref{EQ:Sugiura-trace} to hold. The present paper significantly extends the main results contained in \cite{dr:suffkernel} in different ways.
 
      
To formulate the notions more precisely, let $H_1, H_2$ be complex Hilbert spaces 
 and let $T:H_1\rightarrow H_2$ be a compact operator. Then $T^*T:H_1\rightarrow H_1$ 
 is compact, self-adjoint, and non-negative. Hence, we can define the absolute value of $T$ by the equality $|T|=(T^*T)^\half:H_1\rightarrow H_1$.  Let $(\psi_k)_k$ be an orthonormal basis for $H_1$ consisting of eigenvectors of $|T|$, and let $s_k(T)$ be the eigenvalue corresponding to the eigenvector 
$\psi_k$, $k=1,2,\dots$. The  numbers $s_1(T)\geq s_2(T)\geq \cdots \geq s_n(T)\geq\cdots \geq 0$, are called the singular values of $T:H_1\rightarrow H_2$.  If $0<p<\infty$ and the sequence of singular values is $\ell^p$-summable, then $T$ 
is said to belong to the Schatten-von Neumann class  ${S}_p(H_1, H_2)$. If $1\leq p <\infty$, a norm is associated to ${S}_p(H_1, H_2)$ by
 \[
 \|T\|_{S_p}:=\left(\sum\limits_{k=1}^{\infty}(s_k(T))^p\right)^{\frac{1}{p}}.
 \] 
If $1\leq p<\infty$ 
 the class $S_p(H_1, H_2)$ becomes a Banach space endowed with the norm $\|T\|_{S_p}$. If $p=\infty$ we define $S_{\infty}(H_1, H_2)$ to be the class of compact operators from $H_1$ into $H_2$ endowed  with the operator norm  
$\|T\|_{S_\infty}:=\|T\|_{op}$.
In the case $0<p<1$  the quantity $\|T\|_{S_p}$ only defines
 a  quasinorm, and $S_p(H_1, H_2)$ is also complete. If $H_1=H_2=H$
 we will simply write $S_p(H, H)=S_p(H)$.
 
The Schatten-von Neumann classes are nested, with
\begin{equation}\label{EQ:Sch-nested}
{S}_p\subset {S}_q,\,\,\textrm{ if }\,\, 0<p<q\leq\infty,
\end{equation}
and satisfy the important multiplication property (cf. \cite{weid:hi},  \cite{sim:trace}, \cite{gokr})
\beq\label{inn}{S}_q{S}_p\subset S_r,\eq
where
\beq\label{innf}\frac{1}{r}=\frac{1}{p}+\frac{1}{q},\qquad 0<p<q\leq\infty.\eq
The inclusion \eqref{inn} should be understood in the following sense: If $B\in S_p(H, H_1)$ and $A\in S_q(H_1, H_2)$ then $AB\in S_r(H, H_2)$ provided that \eqref{innf} holds. Moreover, one has 
\beq \label{innscbg}\|AB\|_{S_r(H, H_2)}\leq \|A\|_{S_q(H_1, H_2)} \|B\|_{S_p(H, H_1)}.\eq
The inequality \eqref{innscbg} can be obtained from Theorem 7.8 (c) of \cite{weid:hi},  and  a diagonalisation argument for $A$ and $B$.


\medskip
A relationship between the singular values $s_n(T)$ and the eigenvalues $\lambda_n(T)$ in the case $H=H_1=H_2$ for an operator $T\in S_{\infty}(H)$ was established by Hermann Weyl (cf. \cite{hw:ineq}):

\[\sum\limits_{n=1}^{\infty}|\lambda_n(T)|^p\leq \sum\limits_{n=1}^{\infty} s_n(T)^p ,\quad p>0.\]

We will apply (\ref{inn}) for factorising our operators $T$ in the form $T=AB$ with $A\in S_q$ and $B\in S_p$, 
and from this we deduce that $T\in S_r$. We refer the reader to \cite[Chapter 7]{weid:hi} for more details on the Schatten-von Neumann  classes for operators acting on different Hilbert spaces $H_1, H_2$. Standard references for the special case $H_1=H_2$ are \cite{gokr}, \cite{r-s:mp}, \cite{sim:trace}, \cite{sch:id}. If
 $H=H_1=H_2$ and $T\in S_1(H)$ then $T$ is said to be in the {\em trace class} $S_1$.  Let $T:H\rightarrow H$ be an operator in $S_1(H)$ and let  $(\phi_k)_k$ be any orthonormal basis for $H$. Then, the series $\sum\limits_{k=1}^{\infty} \p{T\phi_k,\phi_k}$   is absolutely convergent and the sum is independent of the choice of the orthonormal basis $(\phi_k)_k$. Thus, we can define the trace $\Tr(T)$ of any linear operator
$T:H\rightarrow H$ in $S_1$ by 
$$
\Tr(T):=\sum_{k=1}^{\infty}\p{T\phi_k,\phi_k},
$$
where $\{\phi_k: k=1,2,\dots\}$ is any orthonormal basis for $H$. If the singular values
are square-summable $T$ is called a {\em Hilbert-Schmidt} operator. It is clear that every trace class operator is a Hilbert-Schmidt operator. A nice basic introduction to the study of the trace class is contained in the book \cite{lax:fa} by Peter Lax.  
 
If $H_j=L^2({\Omega}_j,{\mathcal{M}}_j,\mu_j)$ ($j=1,2$), it is well known that $T$ is a Hilbert-Schmidt operator from  $L^2({\Omega}_1,{\mathcal{M}}_1,\mu_1)$ into $L^2({\Omega}_2,{\mathcal{M}}_2,\mu_2)$ 
if and only if  $T$ can be represented by an integral kernel $K=K(x,y)\in L^2(\Omega_2\times\Omega_1,\mu_2\otimes\mu_1)$. 

We note that in view of \eqref{EQ:Sch-nested} the condition 
$K\in L^2(\Omega_2\times\Omega_1)$ implies that $T\in S_p$ for all $p\geq 2$.
 The situation for Schatten-von Neumann  classes $S_p$ for $p>2$ is indeed simpler and, 
 in fact, similar to that of
 $p=2$. In particular, B. Russo has proved in \cite{br:sch} that for $\mu_1=\mu_2=\mu$ on $\Omega$: If $1< p<2$, $\frac 1p+\frac 1{p'}=1$ and $K\in L^2(\Omega\times\Omega)$ then for the corresponding integral operator $T$ one has 
 \begin{equation}\label{EQ:Russo}
\|T\|_{S_{p'}}\leq\left(\|K\|_{p,p'}\|K^*\|_{p,p'}\right)^{\half}, 
\end{equation} 
 where $\|\cdot\|_{p,q}$ denotes the mixed-norm:
 \begin{equation}\label{EQ:mnorm}
\|K\|_{p,q}=\left(\int_{\Omega}\left(\int_{\Omega}|K(x,y)|^pd\mu(x)\right)^{\frac qp}d\mu(y)\right )^{\frac 1q},
\end{equation} 
and 
$K^*(x,y)=\overline{K(y,x)}$. 
The condition $K\in L^2(\Omega\times\Omega)$ in the above statement can be removed, see
Goffeng \cite{Goffeng}. Namely, even without such assumption, \eqref{EQ:Russo} still holds as long as its right hand side is finite. 
See also the discussion of such topics around \cite[Theorem 2.3]{RRS-JFA}.
 
 \medskip
For $p<2$, the situation is much more subtle, and
the Schatten-von Neumann class $S_p(L^2(\mu_1), L^2(\mu_2))$ cannot be characterised as in the case $p=2$ by a property analogous to the square integrability of integral kernels. This is a crucial fact that we now briefly describe. A classical result of 
Carleman \cite{car:ex} from 1916 gives the construction of a periodic {\em continuous} function 
$\varkappa(x)=\sum\limits_{k=-\infty}^{\infty}c_k e^{2\pi ikx}$ for which the Fourier coefficients $c_k$ satisfy
\begin{equation}\label{EQ:Carleman}
\sum\limits_{k=-\infty}^{\infty} |c_k|^p=\infty\qquad \textrm{ for any } p<2.
\end{equation}
Now, using this and considering the normal operator 
\begin{equation}\label{EQ:Carleman2}
Tf=f*\varkappa
\end{equation}
acting on $L^2(\T)$ 
one obtains that the sequence $(c_k)_k$ forms a complete system of eigenvalues of 
this operator corresponding to the complete orthonormal system 
$$\phi_k(x)=e^{2\pi ikx}, \quad T\phi_k=c_k\phi_k.$$ 
The system $\phi_k$ is also complete for $T^*$, $T^*\phi_k=\overline{c_k}\phi_k$, 
so that the singular values of $T$ are given by $ s_k(T)=|c_k|$, and
hence by \eqref{EQ:Carleman} we have
$$\sum\limits_{k=-\infty}^{\infty}s_k(T)^p=\infty \qquad \textrm{ for any } p<2.$$   
 In other words, in contrast to the case of the class ${S}_2$ of Hilbert-Schmidt operators which is characterised by the square integrability of the kernel, Carleman's result shows that below the index $p=2$ the class of kernels generating operators in the Schatten-von Neumann  class $S_p$  cannot be characterised by criteria of  the type $$ \iint |F(K(x,y))| dxdy<\infty ,$$
for any continuous function $F$ since the kernel $K(x,y)=\varkappa (x-y)$ of the operator $T$ in
 \eqref{EQ:Carleman2} satisfies any kind of integral condition of such form 
 due to the boundedness of $\varkappa$. 

This example demonstrates the relevance of obtaining criteria for operators to belong to Schatten-von Neumann  classes
 for $p<2$ and, in particular, motivates the results in this paper.  Among other things,
 we may also note that the continuity of the kernel (as in the above example)
 also does not guarantee that the operator would belong to any of the Schatten-von Neumann classes $S_p$ with $p<2$. Therefore, it is natural to ask what regularity imposed on the
 kernel would guarantee such inclusions (for example, the $C^2$ condition in Sugiura's
 result mentioned earlier does imply the traceability on $\mathbb T^1$). Thus, these
 questions will be the main interest of the present paper.

The main result for operators to belong to Schatten-von Neumann classes $S_{p}$ for $0<p<2$, is given in Theorem \ref{ext322}. In this work we allow singularities in the kernel so that the formula \eqref{EQ:Sugiura-trace}
would need to be modified in order for the integral over the diagonal to make sense. In such case, in order to calculate the trace of an integral operator using a non-continuous kernel along the diagonal, one idea is to average it to obtain an integrable function. Such an averaging process has been introduced by Weidmann \cite{weid:av} in the Euclidean setting, and applied by Brislawn in \cite{bri:k1}, \cite{bri:k2} 
for integral operators on $L^2(\ern)$ and on $L^2({\Omega},{\mathcal{M}},\mu)$, respectively, where $\Omega$ is a second countable topological space endowed with a $\sigma$-finite Borel measure. The corresponding extensions to the $L^p$ setting have been established in \cite{del:trace} and \cite{del:tracetop}. The $L^2$ regularity of such an averaging process
 is a consequence of the $L^2$-boundedness of the martingale maximal function. Denoting by $\widetilde{K}(x,x)$ the pointwise values of this averaging process, Brislawn  \cite{bri:k2} proved the following formula for a trace class operator $T$ on $L^2(\mu)$, for the extension to $L^p$ see \cite{del:trace}:
\beq\label{f1} \Tr(T)=\int_{\Omega} \widetilde{K}(x,x)d\mu(x).\eq

In Section \ref{SEC:Schatten-classes} we establish our criteria for Schatten-von Neumann  classes on measure spaces, and the special case of the trace class is treated in Section \ref{SEC:trace-class}. 
For this, we briefly recall the definition of the averaging process involved in formula (\ref{f1}).  In Section \ref{SEC:spas}  we present further criteria in terms of operators for which the distribution of eigenvalues is known in terms of the asymptotics of the eigenvalue counting functions. In Section \ref{SEC:appl} we give applications of the obtained conditions in different settings.

\smallskip
The authors would like to thank Alberto Parmeggiani for a discussion.



\section{Schatten-von Neumann classes on $L^2$-spaces}  
\label{SEC:Schatten-classes}

In this section we present our results in the setting of $L^2$-spaces which is not restrictive in terms of the general theory of Hilbert spaces.
 Before stating our first result, we point out that a look at the proof of the trace formula
\eqref{EQ:Sugiura-trace} in \cite[Prop 3.5]{sugiura:book} shows 
that statement can be already improved in the following way:
\begin{prop} \label{ext1} 
Let $\Delta=\frac{\partial ^2}{\partial\theta ^2}+\frac{\partial ^2}{\partial\phi ^2} $ be the Laplacian on $\TT$. 
Let $K(\theta,\phi)$ be a measurable function on $\TT$ and suppose that there exists an integer $q>1$ such that 
$\displaystyle{\Delta^{\frac{q}{2}}}K\in L^2(\T\times\T)$. Then the integral operator $L$ on 
 $L^2(\T)$ defined by
 \[(Lf)(\theta)=\int_0^{2\pi}K(\theta,\phi)f(\phi)d\phi,\]
is trace class and has the trace
\[\Tr(L)=\int_0^{2\pi}\widetilde{K}(\theta,\theta)d\theta,\]
where $\widetilde{K}$ stands for the averaging process described in Section \ref{SEC:trace-class}.
\end{prop}

We observe that for $K\in L^2(\mu_2\otimes\mu_1)$, we have
\[\|K\|_{L^2(\mu_2\otimes\mu_1)}^2=\int_{\Omega_2\times\Omega_1}|K(x,y)|^2d\mu_2(x)\mu_1(y)=\int_{\Omega_1}\left(\int_{\Omega_2}|K(x,y)|^2d\mu_2(x)\right)d\mu_1(y),\]
or we can also write 
\[K\in L^2(\mu_2\otimes\mu_1)\iff K\in L_y^2(\mu_1, L_x^2(\mu_2)) .\]
In particular, this also means that $K_y$ defined by $K_y(x)=K(x,y)$ is well-defined for almost every $y\in\Omega_1$ as a function in  $L_x^2(\mu_2)$. For an operator $E$ from $L^2(\mu_1)$ into $L^2(\mu_1)$ we will use the notation $E_x K(x,y)$ to emphasize that the operator $E$ is acting on the $x$-variable. Analogously, we will also use the notation $E_y K(x,y)$ for an operator $E$ from $L^2(\mu_2)$ into $L^2(\mu_2)$ acting on the $y$-variable.


In this paper we will only consider linear operators.    
We will now give our main criteria for Schatten-von Neumann classes where we note that we do not have to assume the operators $E_1, E_2$ to be self-adjoint nor bounded, i.e. they are considered to be densely defined without further explanations. For a densely defined operator $E$ on a Hilbert, it is well known that, $E$ has a bounded inverse if and only if $E$ is closed and bijective. Henceforth, an {\em invertible operator } is understood as an operator with bounded inverse.   An {\em unbounded operator} is understood as a densely defined operator.    
 
\begin{thm} \label{ext322} 
Let $(\Omega_j,\mathcal{M}_j,\mu_j)$ $(j=1, 2)$ be $\sigma$-finite measure spaces. Let $E_j$ $(j=1,2)$ be unbounded invertible  operators on  $L^2(\mu_j)$ such that $E_j^{-1}\in S_{p_j}(L^2(\mu_j))$ for some $p_j>0$.  Let 
$K\in L^2(\mu_2\otimes\mu_1)$ and let $T$ be the integral operator  from  
 $L^2(\mu_1)$ to $L^2(\mu_2)$ defined by
 \[(Tf)(x)=\int_{\Omega_1} K(x,y)f(y)d\mu_1(y).\]
Then the following holds:
\begin{enumerate}
\item[(i)] If $(E_2)_x(E_1)_yK\in L^2(\mu_2\otimes\mu_1)$, then $T$ belongs to the Schatten-von Neumann classes $S_r(L^2(\mu_1), L^2(\mu_2))$ for all $0<r<\infty$ such that
\[\frac 1r\leq\half+\frac 1{p_1}+\frac 1{p_2}.\]
Moreover,
\beq\label{par1}\|T\|_{S_r}\leq \|E_1^{-1}\|_{S_{p_1}}\|E_2^{-1}\|_{S_{p_2}}\|(E_2)_x(E_1)_yK\|_{L^2(\mu_2\otimes\mu_1)}.\eq

\item[(ii)] If $(E_2)_xK\in L^2(\mu_2\otimes\mu_1)$, then $T$ belongs to the Schatten-von Neumann classes $S_r(L^2(\mu_1), L^2(\mu_2))$ for all $0<r<\infty$ such that
\[\frac 1r\leq \half+\frac 1{p_2}.\]
Moreover,
\beq\label{par2}\|T\|_{S_r}\leq  \|E_2^{-1}\|_{S_{p_2}}\|(E_2)_xK\|_{L^2(\mu_2\otimes\mu_1)}.\eq

\item[(iii)] If $(E_1)_yK\in L^2(\mu_2\otimes\mu_1)$, then $T$ belongs to the Schatten-von Neumann classes $S_r(L^2(\mu_1), L^2(\mu_2))$ for all $0<r<\infty$ such that
\[\frac 1r\leq \half+\frac 1{p_1}.\]
Moreover,
\beq\label{par3}\|T\|_{S_r}\leq \|E_1^{-1}\|_{S_{p_1}}\|(E_1)_yK\|_{L^2(\mu_2\otimes\mu_1)}.\eq

\end{enumerate}
\end{thm}

\begin{rem} \label{REM:L2}
The condition that $K\in L^2(\mu_2\otimes\mu_1)$ in Theorem \ref{ext322} is not restrictive.
Indeed, conditions for $T\in S_r(L^2(\mu_1), L^2(\mu_2))$ for $r>2$ do not require regularity of $K$ and are given, for example, in \eqref{EQ:Russo}. The case $0<r<2$ is much more subtle (as the classes become smaller), but
if $T\in S_r(L^2(\mu_1), L^2(\mu_2))$ for $0<r<2$ then, in particular, $T$ is a Hilbert-Schmidt operator, and hence the condition $K\in L^2(\mu_2\otimes\mu_1)$ is actually necessary.

The statement of Theorem \ref{ext322} covers precisely the case $0<r<2$. Indeed, for example in Part (i), we have $r=\frac{2p_1p_2}{p_1p_2+2(p_1+p_2)}$ and hence we have $0<r<2$ 
since in general $0<p_1, p_2<\infty .$ Thus, Theorem \ref{ext322} provides
a sufficient condition for Schatten-von Neumann classes $S_r$ for $0<r<2$. 
\end{rem}

\begin{proof}[Proof of Theorem \ref{ext322}] 

(i) For the sake of simplicity, sometimes we will abbreviate the notation also in integrals by writing $E_1=(E_1)_y$ and $E_2=(E_2)_x$. \\

We now consider the operator $A$ with kernel 
\[A(x,y)=(E_2)_x(E_1)_yK\in L^2(\mu_2\otimes\mu_1).\]
 One can see that
\beq A=E_2\circ T\circ {E_1^*}.\label{refnewAker}\eq 
Since $A\in S_2(L^2(\mu_1),L^2(\mu_2))$ and using the fact that $({E_1^*})^{-1}\in S_{p_1}(L^2(\mu_1))$ if  and only if $E_1^{-1}\in S_{p_1}$ with equal norms, we obtain 
\[T=E_2^{-1}\circ A\circ ({E_1^*})^{-1} \in S_{p_2}\circ S_2\circ S_{p_1}\subset S_r,\]
provided%
\[\frac 1r\leq\half+\frac 1{p_1}+\frac 1{p_2},\]
by \eqref{innf}.

Moreover, for the  estimation of the Schatten-von Neumann norm $\|T\|_{S_r}$, according to \eqref{innscbg} we obtain:   

\[\|T\|_{S_r}=\|E_2^{-1}\,\circ A\,\circ ({E_1^*})^{-1}\|_{S_r} \leq \|E_1^{-1}\|_{S_{p_1}}\|E_2^{-1}\|_{S_{p_2}}\|(E_2)_x(E_1)_yK\|_{L^2(\mu_2\otimes\mu_1)},\]

and this concludes the proof of (i).\\


(ii) Just consider $E_2$ in the proof of (i). \\

(iii) Just consider $E_1$ in the proof of (i).
\end{proof}
\begin{rem}\label{remabsre7} (i) In the general setting of a separable Hilbert spaces $H$  one can construct operators $E$ satisfying the assumptions in the above theorem. Let $0<p<\infty$ and $(s_n)_n$ a sequence in $\ell^p$, with $s_n>0$ for all $n$. Let $\{\phi_n:n=1,2,\dots\}$ be an orthonormal basis of $H$.  \\

We define \[\mathcal{D}:=\{f\in H: \sum\limits_{n=1}^{\infty}|(f,\phi_n)_{H}|^2s_n^{2p}<+\infty\}.\]
Since  $Span(\{\phi_n:n=1,2,\dots\})\subset \mathcal{D}$, then $\mathcal{D}$ is dense in  
  $H$. We define $E\phi_n=s_n^{-p}\phi_n$. By using Cauchy-Schwarz inequality, we can see that $E$ can be extended to $\mathcal{D}$ by $Ef=\sum\limits_{n=1}^{\infty}(f,\phi_n)_{H}s_n^{p}\phi_n$. It is clear that $E$ has a bounded inverse determined by $E^{-1}\phi_n=s_n^{p}\phi_n$ and $E^{-1}\in S_p(H)$. We also note that since $\lim_ns_n^{p}=0$, we have $\lim_ns_n^{-p}=+\infty$ and  the operator $E$ is not bounded.\\
  
In more specific cases, we will consider more concrete operators for the applications.\\

(ii) We point out that a converse statement also holds for the multiplication property \eqref{inn} (cf. \cite[Theorem 7.9]{weid:hi}): Let $0<p,q, r<\infty$ and $T\in S_r(H, H_2)$ with 
\beq\label{innfww}\frac{1}{r}=\frac{1}{p}+\frac{1}{q}.\eq
Then there exist operators $B\in S_p(H, H_1)$ and $A\in S_q(H_1, H_2)$ (with some Hilbert space $H_1$) for which $T=AB$; the operators $A, B$ can be chosen such that $\|T\|_{S_r(H, H_2)}=\|A\|_{S_q(H_1, H_2)}\|B\|_{S_p(H, H_1)}$.
\end{rem}

\begin{rem} \label{EQ:fact}
Under conditions of Theorem \ref{ext322}, in the proof of its Part (i) the main point was to obtain the factorisation 
\begin{equation}\label{EQ:factmain}
T=E_2^{-1} A({E_1^*})^{-1},
\end{equation} 
where $A:L^2(\mu_1)\to L^2(\mu_2)$ is the integral operator with the integral kernel 
$A(x,y)=E_2 E_1 K(x,y).$
This factorisation has other consequences. For example, the combination of \eqref{EQ:factmain}, the condition \eqref{EQ:Russo} and the multiplication property imply the following extension of Theorem \ref{ext322} in the case $\mu_1=\mu_2=\mu$ on $\Omega$, where we will denote by 
$L^{q'}(\Omega, L^q(\Omega))$ the space defined by the mixed norm \eqref{EQ:mnorm}, that is, by
 \begin{equation}\label{EQ:mnorm2}
\|K\|_{L^{q'}(\Omega, L^q(\Omega))}=\left(\int_{\Omega}\left(\int_{\Omega}|K(x,y)|^ q d\mu(x)\right)^{\frac{q'}q}d\mu(y)\right )^{\frac{1}{q'}}<\infty.
\end{equation} 
We also use the notation $K^*(x,y):=\overline{K(y,x)}$.
\end{rem}

\begin{cor}\label{COR:extmain}
Let $(\Omega,\mathcal{M},\mu)$ be a $\sigma$-finite measure space. Let $T$ be a bounded  integral operator on   
 $L^2(\Omega)$, defined by
 \[(Tf)(x)=\int_{\Omega} K(x,y)f(y)d\mu(y).\]
Let $1<q\leq 2$ and $\frac1q+\frac{1}{q'}=1$. Then the following holds:
\begin{enumerate}
\item[(i)] Let $E_1,E_2$ be  unbounded invertible  operators on  $L^2(\Omega)$ such that $E_j^{-1}\in S_{p_j}(L^2(\Omega))$ for some $p_j>0$, $(j=1,2)$.  
If $(E_2)_x(E_1)_yK$ and $((E_2)_x(E_1)_yK)^*\in L^{q'}(\Omega, L^q(\Omega))$, then $T$ belongs to the Schatten-von Neumann classes $S_r(L^2(\Omega))$ for all $0<r<\infty$ such that
\[\frac 1r\leq \frac{1}{q'}+\frac 1{p_1}+\frac 1{p_2}.\]
Moreover,
\begin{multline}\label{pi1a}
\|T\|_{S_r}\leq 
\|E_1^{-1}\|_{S_{p_1}}\|E_2^{-1}\|_{S_{p_2}}\times \\
\times \left(\|(E_2)_x(E_1)_yK\|_{L^{q'}(\Omega, L^q(\Omega))}\|((E_2)_x(E_1)_yK)^*\|_{L^{q'}(\Omega, L^q(\Omega))}\right)^{\half}.
\end{multline} 

\item[(ii)] Let $E$ be an  unbounded invertible  operator on  $L^2(\Omega)$ such that $E^{-1}\in S_{p}(L^2(\Omega))$ for some $p>0$.
If $E_xK ,(E_xK)^*\in L^{q'}(\Omega, L^q(\Omega))$ or $E_yK ,(E_yK)^*\in L^{q'}(\Omega, L^q(\Omega))$, then $T$ belongs to the Schatten-von Neumann classes $S_r(L^2(\Omega))$ for all $0<r<\infty$ such that
\[\frac 1r\leq \frac{1}{q'}+\frac 1{p}.\]
Moreover, respectively one has
\beq \|T\|_{S_r}\leq 
\|E^{-1}\|_{S_{p}}\left(\|E_xK\|_{L^{q'}(\Omega, L^q(\Omega))}\|(E_xK)^*\|_{L^{q'}(\Omega, L^q(\Omega))}\right)^{\half},\label{pi1b}\eq
or
\beq \|T\|_{S_r}\leq 
\|E^{-1}\|_{S_{p}}\left(\|E_yK\|_{L^{q'}(\Omega, L^q(\Omega))}\|(E_yK)^*\|_{L^{q'}(\Omega, L^q(\Omega))}\right)^{\half},\label{pi1c}\eq
respectively.\\

\end{enumerate}
\end{cor} 
Since $L^{2}(\Omega, L^2(\Omega))=L^{2}(\Omega\times\Omega)$, Corollary \ref{COR:extmain} indeed is an extension of Theorem \ref{ext322} in the case of operators acting on the same space $L^2(\Omega)$. 

\begin{rem}\label{REM:Lpq}
We note that following the remarks after \eqref{EQ:mnorm} we do not need to assume in Corollary \ref{COR:extmain}  that $K\in L^2(\Omega\times\Omega)$.
Consequently, compared with the sufficient condition \eqref{EQ:Russo} by Russo and with 
Theorem \ref{ext322}, the Schatten-von Neumann  class index in Corollary \ref{COR:extmain}  can be larger than $2$. Indeed, compared with the argument in Remark \ref{REM:L2}, the condition on $r$ in Corollary \ref{COR:extmain} becomes $0<r<\frac{p_1 p_2}{p_1+p_2}$ in Part (i) and $0<r<p$ in Part (ii), respectively. Therefore, even for Schatten-von Neumann  classes $S_r$ with $r>2$, Corollary \ref{COR:extmain}
extends the sufficient condition \eqref{EQ:Russo} by Russo in the following sense:
  For an integral operator to belong to the Schatten-von Neumann  classes $S_r$ with $r>2$, the `size' condition \eqref{EQ:Russo} can be relaxed if we know that the integral kernel of an integral operator has additional `regularity' properties.
\end{rem} 


\section{Trace class operators and their traces}
\label{SEC:trace-class}

In this section we consider the important case of the trace class operators. We start by deducing a corollary of Theorem \ref{ext322} in this special case.
 In order to establish a formula for the trace we will require an additional topological structure on $\Omega$. We will now briefly recall the averaging process which is required for the study of trace formulae for kernels with discontinuities along the diagonal. We start by defining the martingale
maximal function. Let $({\Omega},{\mathcal{M}},\mu)$ be a $\sigma$-finite measure space and
let $\{\mathcal{M}_j\}_{j}$ be a sequence of sub-$\sigma$-algebras such that
$$
\mathcal{M}_j\subset\mathcal{M}_{j+1}\,\,\textrm{ and }{\mathcal{M}}=\bigcup\limits_{j}\mathcal{M}_j.
$$
In order to define conditional expectations we assume that $\mu$ is $\sigma$-finite on each $\mathcal{M}_j$. In that case, 
if $f\in L^p(\mu)$, then $E(f|\mathcal{M}_n)$ exists. We say that a sequence $\{f_j\}_{j}$ of functions on $\Omega$ is a {\em martingale} if each $f_j$ is $\mathcal{M}_j$-measurable and
\beq E(f_j|\mathcal{M}_k)=f_k\,\mbox{ for }k<j.\eq
 In order to obtain a generalisation of the Hardy-Littlewood maximal function we consider the particular case of martingales generated by a single $\mathcal{M}$-measurable function $f$. The {\em martingale maximal function} is defined by
 \beq 
 Mf(x):=\sup\limits_{j}E(|f|\left .\right|\mathcal{M}_j)(x).
 \eq
This martingale can be defined, in particular, when the $\sigma$-algebra $\mathcal{M}$  
 is countably generated and it will allow to study the trace by mean of an averaging process on the diagonal of the kernel. However, this process is most effective for the computations in the case of a $\sigma$-algebra of Borel sets for a second countable topological space. Henceforth we will assume that $\Omega$ is a second countable topological space, $\mathcal{M}$ is the $\sigma$-algebra of Borel sets and $\mu$ is a $\sigma$-finite Borel measure. For our purposes in the study of the kernel the sequence 
of $\sigma$-algebras is constructed from a corresponding increasing sequence of 
partitions $\mathcal{P}_j\times\mathcal{P}_j$ of $\Omega\times\Omega$. 

Now, for each $(x,y)\in \Omega\times \Omega$ there is a unique $C_j(x)\times C_j(y)\in  \mathcal{P}_j\times\mathcal{P}_j$ containing $(x,y)$. Those sets $C_j(x)$ replace the cubes in $\ern$ in the definition of the classical Hardy-Littlewood maximal function. We refer to Doob 
\cite{doob:book1} for more details on the martingale maximal function
and its properties.\\

We denote by $A_j^{(2)}$ the averaging operators on $\Omega\times\Omega$:  Let $K\in
L_{loc}^1(\mu\otimes\mu)$, then the averaging $A_j^{(2)}$ is defined $\mu\otimes\mu$-almost everywhere (cf. \cite{bri:k2}) by 
\beq 
A_j^{(2)}K(x,y):=\frac{1}{\mu(C_j(x))\mu(C_j(y))}\int\limits_{C_j(x)}\int\limits_{C_j(y)}K(s,t)d\mu(t)d\mu(s).
\eq

The averaging process will be applied to the kernels $K(x,y)$ of our operators.  
As a consequence of the fundamental properties of the martingale maximal function it can be deduced that 
\begin{equation}\label{EQ:Kt}
\widetilde{K}(x,y):=\lim_{j\rightarrow\infty} A_j^{(2)}K(x,y)
\end{equation} 
is defined almost everywhere and that it agrees with $K(x,y)$ in the points of continuity.  
 We observe that if $K(x,y)$ is the integral kernel of a trace class operator, 
 then $K(x,y)$ is, in particular, square integrable on $\Omega\times\Omega.$ 
\medskip
A classical example with a discontinuous kernel is the Volterra 
operator $V$ on $L^2(I)$ where $I=[0,1]$. Its kernel is given by
\begin{equation*}
 K(x,y)=\left\{
\begin{array}{rl}
1\,\,\,\,;& y\leq x,\\
0\,\,\,\,;&x< y.
\end{array} \right.
\label{volt12}\end{equation*}
By averaging on cubes one can see that $\widetilde{K}(x,x)=\half$ for $0<x<1$. However,
 it is well known that its singular values are $s_n=2(\pi(2n+1))^{-1}$, hence $V$ is not a trace class operator.\\

In the sequel in this section, we can always assume that $K\in L^2(\mu\otimes\mu)$ since
it is not restrictive because the trace class is included in the Hilbert-Schmidt class, and
the square integrability of the kernel is then a necessary condition. 

As usual, we are using the notation $L^2(\mu)\equiv L^2(\Omega)$.


\begin{cor} \label{ext32cor} 
Let 
 $(\Omega,\mathcal{M},\mu)$ be a measure space endowed with  a $\sigma$-finite measure $\mu$. Let $E_j$ $(j=1,2)$ be   unbounded  invertible operators on  $L^2(\Omega)$  such that $E_j^{-1}\in S_{p_j}(L^2(\Omega))$ for some $p_j>0$.  Let 
$K\in L^2(\Omega\times\Omega)$ and let $T$ be the integral operator  from  
 $L^2(\Omega)$ to $L^2(\Omega)$ defined by
 \[(Tf)(x)=\int_{\Omega} K(x,y)f(y)d\mu(y).\]
Let $1<q\leq 2$ and $\frac1q+\frac{1}{q'}=1$.
\begin{enumerate}
\item[(i)] 
If $(E_2)_x(E_1)_yK$ and $((E_2)_x(E_1)_yK)^*\in L^{q'}(\Omega, L^q(\Omega))$, then $T$ belongs to the trace class $S_1(L^2(\mu))$ provided that
\[1\leq \frac{1}{q'}+\frac 1{p_1}+\frac 1{p_2}.\]
Moreover, we have
\begin{multline}\label{pi1atr}
 \|T\|_{S_1}\leq \|E_1^{-1}\|_{S_{p_1}}\|E_2^{-1}\|_{S_{p_2}}\times
 \\ \times \left(\|(E_2)_x(E_1)_yK\|_{L^{q'}(\Omega, L^q(\Omega))}\|((E_2)_x(E_1)_yK)^*\|_{L^{q'}(\Omega, L^q(\Omega))}\right)^{\half}.
\end{multline}

In particular, if $(E_2)_x(E_1)_yK\in L^2(\Omega\times\Omega)$, then 
$T$ belongs to the trace class $S_1(L^2(\Omega))$ provided that
$\half=\frac 1{p_1}+\frac 1{p_2}.$ 
\item[(ii)] 
Let $E$ be an  unbounded invertible  operator on  $L^2(\Omega)$ such that $E^{-1}\in S_{p}(L^2(\Omega))$ for some $p>0$.
If $E_xK ,(E_xK)^*\in L^{q'}(\Omega, L^q(\Omega))$ or $E_yK ,(E_yK)^*\in L^{q'}(\Omega, L^q(\Omega))$, then $T$ belongs to the trace class $S_1(L^2(\mu))$ provided that
\[1\leq \frac{1}{q'}+\frac 1{p}.\]
Moreover, respectively one has
\beq \|T\|_{S_1}\leq \|E^{-1}\|_{S_{p}}\left(\|E_xK\|_{L^{q'}(\Omega, L^q(\Omega))}\|(E_xK)^*\|_{L^{q'}(\Omega, L^q(\Omega))}\right)^{\half},\label{pi1bb}\eq
or
\beq \|T\|_{S_1}\leq \|E^{-1}\|_{S_{p}}\left(\|E_yK\|_{L^{q'}(\Omega, L^q(\Omega))}\|(E_yK)^*\|_{L^{q'}(\Omega, L^q(\Omega))}\right)^{\half}.\label{pi1cc}\eq

In particular, if $E$ is an  unbounded  invertible operator on  $L^2(\Omega)$  such that $E^{-1}\in S_{2}(L^2(\Omega))$ and either $E_yK\in L^2(\Omega\times\Omega)$ or 
 $E_xK\in L^2(\Omega\times\Omega)$, then $T$ belongs to the trace class $S_1(L^2(\Omega))$. 

\item[(iii)] Moreover, assume additionally that $\Omega$ is a second countable topological space and  $(\Omega,\mathcal{M},\mu)$ is a measure space endowed with  a $\sigma$-finite Borel measure $\mu$. Then under any of the assumptions (i) or (ii), the operator $T$ is trace class on $L^2(\mu)$ and 
  its trace is given by
\begin{equation}\label{EQ:trace}
\Tr(T)=\int_{\Omega}\widetilde{K}(x,x)d\mu(x). 
\end{equation}
In particular, if $K$ is continuous on the diagonal one has
\beq\Tr(T)=\int_{\Omega}K(x,x)d\mu(x).\label{EQ:tracec}\eq
\end{enumerate}
\end{cor}

\begin{proof} By taking $r=1$ in the corresponding assumptions in Corollary \ref{COR:extmain} one can deduce (i) and (ii). For (iii) the fact that $T$ is trace class follows from the corresponding assumption (i) or (ii), and the  trace formula comes from (\ref{f1}), with $\widetilde{K}$ given by \eqref{EQ:Kt}. The last part, follows since $\widetilde{K}$ agrees with $K$ in the points of continuity.
\end{proof}

\begin{rem}\label{REM:Lidski}
(a) Combining the statement of Part (iii) of Corollary \ref{ext32cor} with the celebrated Lidskii formula  \cite{li:formula} we can extend the trace formula \eqref{EQ:trace} in Part (iii) by
\begin{equation}\label{EQ:trace2}
\Tr(T)=\int_{\Omega}\widetilde{K}(x,x)d\mu(x)=\sum_j \lambda_j,
\end{equation}
where $\lambda_j$ are the eigenvalues of the operator $T$ counted with multiplicities.\\


\noindent (b) The additional assumption on $\Omega$ to be a second countable topological space 
is only required in order to obtain the additional formula \eqref{EQ:tracec}. This requirement is enough general for the applications we will consider in this work.
\\

\noindent (c) If $\Omega$ is a second topological space and $K\in L^2(\mu\otimes\mu)$ we have $\tilde{K}(x,y)=K(x,y)$ for the points of continuity of $K$. Hence any continuous kernel on the diagonal provides an example where this limit can be obtained just as the pointwise value $K(x,x)$. An example of relevance in spectral geometry is provided in 
 Remark \ref{remsgeom4} (b) with the kernel of the double layer potential for a $C^2$  bounded region $\Omega$ in $\er^2$. Indeed, it is known from two dimensional potential theory that this kernel is continuous and $K(x,x)=-\half\kappa(x)$ where $\kappa(x)$ is the curvature of $\partial\Omega$ at $x$. 


\end{rem}

\section{Conditions in terms of spectral asymptotics}
\label{SEC:spas}

The typical application of the results above may come from the observation that knowing the spectral asymptotics of $E_1$ and $E_2$ implies conclusions about the membership in Schatten-von Neumann classes for their inverses. However, in the case when the spectral asymptotics of operators $E_1$ and $E_2$ are available, the spectral conclusions for the integral operators can be sharpened further in terms of the decay rates of their singular numbers.

As further examples, in Section \ref{SEC:appl} we will consider 
different kinds of domains and operators to test the membership in the Schatten-von Neumann classes.

The following conditions are based on the knowledge of the behaviour of the eigenvalue counting function of the operators $E_1, E_2$. We recall that for a self-adjoint operator $E$ with discrete spectrum $\{\lambda_j\}_j$ its eigenvalue counting function is defined by
$$
N(\lambda):=\#\{j: \lambda_j\leq \lambda\},
$$
where $\lambda_j$'s are counted with their respective multiplicities.
The conditions that we
 will impose can be effectively verified as we will shown in the subsequent subsections. 

\begin{thm} \label{extahk} 
Let $(\Omega_i,\mathcal{M}_i,\mu_i)$ $(i=1,2)$ be $\sigma$-finite measure spaces. For each $i=1,2$, let $E_i$ be an essentially self-adjoint operator on $L^2(\mu_i)$ such that the spectrum of its closure consists of a sequence of discrete and strictly positive eigenvalues $0<\lambda_{1,i}\leq\lambda_{2,i}\leq\cdots ,$ whose eigenvectors are a basis of $L^2(\mu_i)$. Assume that for the eigenvalue counting function $N_{i}(\lambda)$ of $E_i\,\,(i=1,2)$ there exist
  constants $C_i, p_i>0$ such that 
\beq\label{nilq} N_i(\lambda)\leq C_i(1+\lambda)^{p_i} \mbox{ for all }\lambda >0.\eq
Let $K\in L^2(\mu_2\otimes\mu_1)$ and let $T$ be the integral operator  from  
 $L^2(\mu_1)$ to $L^2(\mu_2)$ defined by
 \[(Tf)(x)=\int_{\Omega_1} K(x,y)f(y)d\mu_1(y).\]
Then the following holds:
\begin{enumerate}
\item[(i)] If $(E_2)_x(E_1)_yK\in L^2(\mu_2\otimes\mu_1)$, then $T$ belongs to the Schatten-von Neumann class $S_r(L^2(\mu_1),L^2(\mu_2))$ for all  $0<r<\infty$ such that 
\[\frac 1r<\half+\frac{1}{p_1}+\frac{1}{p_2},\]
and \eqref{par1} holds. \\

Moreover, the sequence of singular values $(s_k(T))_k$ satisfies the following
estimate for the rate of decay:
\[s_k(T)=o(k^{-\left(\half+\frac{1}{p_1}+\frac{1}{p_2}\right)}).\]

\item[(ii)] Let $E$ be an  unbounded invertible operator on  $L^2$ as above such that its spectrum satisfies  \eqref{nilq}  for  some $p>0$.  If either $E_yK\in L^2(\mu_2\otimes\mu_1)$ or  $E_xK\in L^2(\mu_2\otimes\mu_1)$, then $T$ belongs to the Schatten-von Neumann class $S_r(L^2(\mu_1),L^2(\mu_2))$ for  all  $0<r<\infty$ such that
\[\frac 1r<\half+\frac 1p,\]
and respectively \eqref{par2} or \eqref{par3} holds.\\

Moreover, the sequence of singular values $(s_k(T))_k$ satisfies the following
estimate for the rate of decay:
\[s_k(T)=o(k^{-\left(\half+\frac 1p\right)}).\]
\end{enumerate}
\end{thm}
\begin{proof} (i) We note that the assumptions on $N_i$ for $i=1,2$, imply that 
\[k=N(\lambda_{k,i})\leq C_i\lambda_{k,i}^{p_i} .\]
Hence 
\beq\label{sejx} k^{\frac{1}{p_i}}\lambda_{k,i}^{-1}\leq C_i^\prime\eq
and thus also
\beq\label{sejex}\sum\limits_{k=1}^{\infty}\lambda_{k_i}^{-q_i}<\infty, \,\,\mbox{ for all }\,\,q_i>p_i.\eq
Thus $E_i^{-1}$ is a compact operator and its singular values are $s_k(E_i^{-1})=\lambda_{k,i}^{-1}$ and $E_i^{-1}\in S_{q_i}(L^2(\mu_i))$ for all $q_i>p_i$. Now, for  $q_i>p_i$  the fact that $T\in S_r(L^2(\mu_1),L^2(\mu_2))$ can now be deduced from Theorem \ref{ext322}   and    
 \[\half+\frac{1}{q_1}+\frac{1}{q_2}=\frac 1r<\half+\frac{1}{p_1}+\frac{1}{p_2}.\]

In order to get the estimate for the rate of decay of the singular values we will use the
 following  Fan's inequality (cf. \cite{fan:sch}, \cite{sim:trace}) for the singular values of the composition of two compact operators:
\beq\label{fan1} s_{k+l-1}(BC)\leq s_{k}(B)s_{l}(C),\eq
for all $k,l\geq 1$.

We will apply \eqref{fan1} to the factorisation $T=E_2^{-1}A(\overline{E_1^*})^{-1}$ obtained in the proof of Theorem \ref{ext322}.
By using \eqref{fan1} with $l+m-1$ instead of $l$ we get
\[s_{k+l+m-2}(T)\leq s_{k}(E_2^{-1})s_{l+m-1}(A(\overline{E_1^*})^{-1})\leq s_{k}(E_2^{-1})s_{l}(A)s_{m}(E_1^{-1}), \]
for $k,l,m\geq 1.$

Thus, with $k=l=m$ we obtain
\[s_{3k-2}(T)\leq s_{k}(E_2^{-1})s_{k}(A)s_{k}(E_1^{-1}).\]
Hence and by \eqref{sejx} we have
\begin{align*} \sum\limits_{k=1}^{\infty}k^{2(\frac 1{p_1}+\frac 1{p_2})}s_{3k-2}(T)^2\leq &\sum\limits_{k=1}^{\infty}k^{\frac 2{p_2}}s_{k}(E_2^{-1})^2s_{k}(A)^2k^{\frac 2{p_1}}s_{k}(E_1^{-1})^2 \\
 \leq &(C_1C_2)^2 \sum\limits_{k=1}^{\infty}s_{k}(A)^2<\infty .
 \end{align*}
Since $(s_{k}(T))_k$ is a non-increasing sequence, then  $s_{3k}(T), s_{3k-1}(T)\leq s_{3k-2}(T)$ and  
\[\sum\limits_{k=1}^{\infty}k^{2(\frac 1{p_1}+\frac 1{p_2})}s_{k}(T)^2<\infty .\]
Therefore
\[s_k(T)=o(k^{-\frac{1}{\tau}}),\]
where $\tau=(\half+\frac 1{p_1}+\frac 1{p_2})^{-1}.$ This concludes the proof of (i).

The proof of (ii) follows in a similar way by considering the factorisation $T=E^{-1}A$ so we can omit the details.
\end{proof}

\begin{rem}\label{REM:Next}
If $\Omega_1=\Omega_2=\Omega$ and $\mu_1=\mu_2=\mu$, the statement of Theorem \ref{extahk} can be extended by  using Corollary \ref{COR:extmain} instead of Theorem \ref{ext322}. More precisely, assume that $E_1$ and $E_2$ satisfy the assumptions of 
Theorem \ref{extahk}. 
Let $1<q\leq 2$ and $\frac1q+\frac{1}{q'}=1$.
Then the following holds:
\begin{enumerate}
\item[(i)] If $(E_2)_x(E_1)_yK$ and $((E_2)_x(E_1)_yK)^*\in L^{q'}(\Omega, L^q(\Omega))$, then $T$ belongs to the Schatten-von Neumann class $S_r(L^2(\Omega))$ for all  $0<r<\infty$ such that 
\[\frac 1r<\frac{1}{q'}+\frac{1}{p_1}+\frac{1}{p_2},\]
and \eqref{pi1a} holds.\\

Moreover, the sequence of singular values $(s_k(T))_k$ satisfies the following
estimate for the rate of decay:
\[s_k(T)=o(k^{-\left(\frac{1}{q'}+\frac{1}{p_1}+\frac{1}{p_2}\right)}).\]

\item[(ii)] Let $E$ be an  unbounded  invertible operator on  $L^2$ as above such that its spectrum satisfies  \eqref{nilq}  for  some $p>0$.  
If $E_xK ,(E_xK)^*\in L^{q'}(\Omega, L^q(\Omega))$ or $E_yK ,(E_yK)^*\in L^{q'}(\Omega, L^q(\Omega))$, then $T$ belongs to the Schatten-von Neumann class $S_r(L^2(\Omega))$ 
 for  all  $0<r<\infty$ such that
\[\frac 1r<\frac{1}{q'}+\frac 1p,\]
and respectively \eqref{pi1b} or \eqref{pi1c} holds.\\

Moreover, the sequence of singular values $(s_k(T))_k$ satisfies the following
estimate for the rate of decay:
\[s_k(T)=o(k^{-\left(\frac{1}{q'}+\frac 1p\right)}).\]
\end{enumerate}
\end{rem} 

\section{Applications}
\label{SEC:appl}

In this section we will describe several example situations where one can apply the obtained results:

\begin{itemize}
\item compact manifolds: taking $E_1, E_2$ to be elliptic pseudo-differential operators one obtains conditions in terms of the regularity of the kernel;
\item lattices: here the regularity of the kernel becomes irrelevant; however, due to non-compactness the conditions are formulated in terms of the behaviour of the integral kernel at infinity;
\item $\Rn$: for domains which are not necessarily bounded but have finite Lebesgue measure in Section \ref{SEC:Riesz} we obtain conditions still {\em only} in terms of the regularity of the kernel;
\item $\Rn$: in general, due to non-boundedness the regularity of the kernel by itself is not sufficient to ensure the compactness of the operator, and the regularity assumptions should be combined with decay conditions at infinity. It is convenient to formulate such conditions in terms of the action of harmonic or anharmonic oscillators on the kernel; in particular, it shows that {\em different combinations of regularity and decay} may ensure the membership in the Schatten-von Neumann classes on $\Rn$;
\item sub-Riemannian settings: here is may be natural to formulate the conditions in terms of the operators associated to the sub-Riemannian structure (such as the sub-Laplacian). In Section \ref{SEC:subR} we briefly discuss the implications for general compact sub-Riemannian manifolds, contact manifolds, strictly pseudo-convex CR manifolds, and (sub-)Laplacians on compact Lie groups.
\end{itemize} 

Thus, in the following subsections we consider several applications of Theorem \ref{ext322}, Corollary \ref{ext32cor} and Theorem \ref{extahk}.

In the case when the operators act on the same space we also have natural extensions of the statements below by using mixed $L^{q'}(\Omega, L^q(\Omega))$ norms as in 
Corollary \ref{COR:extmain} and Remark \ref{REM:Next} instead. For simplicity, we mostly restrict to the $L^2$-case since the extensions to the $L^{q'}(\Omega, L^q(\Omega))$ setting are rather straightforward.

\subsection{Operators on closed manifolds}
\label{SEC:appcm}

In this section we will consider the case of integral operators
on a compact manifold without boundary.

Thus, let $M$ be a smooth compact manifold without boundary of dimension $n$ endowed with a volume element $dx$.
 We denote by $\Psi^{\nu}_{+e}(M)$ the class of positive elliptic pseudo-differential 
 operators of order $\nu\in\er$, 
 i.e. positive operators which in every coordinate chart are operators in H\"ormander classes 
 on $\Rn$ with elliptic symbols
 in $S^\nu_{1,0}$, see e.g. \cite{shubin:r}.
 
We note that for any positive elliptic operator $P\in \Psi^\nu_{+e}(M)$ the standard Sobolev space $H^\mu(M)$ defined in local coordinates can be characterised
as the space of all distributions $f\in\mathcal D'(M)$ such that
$(I+P)^{\frac\mu\nu} f\in L^2(M)$.

Let now $M_1, M_2$ be closed manifolds and $P_i\in\Psi^{\nu_i}_{+e}(M_i)$
($i=1,2$) with $\nu_i>0$. Consequently, the following mixed regularity Sobolev space 
$H^{\mu_2,\mu_1}_{x,y}(M_2\times M_1)$
of mixed regularity $\mu_1,\mu_2\geq 0$, defined by
\begin{equation}\label{EQ:mixed-Sobolev}
K\in H^{\mu_2,\mu_1}_{x,y}(M_2\times M_1) \Longleftrightarrow
(I+P_2)_x^{\frac{\mu_2}{\nu_2}}(I+P_1)_y^{\frac{\mu_1}{\nu_1}}  K \in L^2(M_2\times M_1),
\end{equation}
is independent of the choice of operators $P_1, P_2$.

The relation between these mixed Sobolev spaces and the standard Sobolev spaces $H^\mu(M_2\times M_1)$ on the manifold $M_2\times M_1$ is given by
\begin{equation}\label{EQ:Sobs}
H^{\mu_1+\mu_2}(M_2\times M_1)\subset H^{\mu_2,\mu_1}_{x,y}(M_2\times M_1)\subset
H^{\min(\mu_1,\mu_2)}(M_2\times M_1),
\end{equation} 
for all $\mu_1,\mu_2\geq 0$. This can be readily seen by an extension of an argument in \cite[Proposition 4.3]{dr:suffkernel} where this was shown to hold in the case of $M_1=M_2$.


Then we have the following statement. We will write $E_i=(I+P_i)^{\frac{\mu_i}{\nu_i}} $ for $i=1,2$.
\begin{cor} \label{ext322l} 
Let $M_1, M_2$ be closed manifolds of dimensions $n_1, n_2$, respectively, and let $\mu_1, \mu_2 \geq 0$. 
Let 
$K\in L^2(M_2\times M_1)$ be such that $K\in H^{\mu_2,\mu_1}_{x,y}(M_2\times M_1)$. Then
the integral operator $T$ from $L^2(M_1)$ to $L^2(M_2)$ defined by
 \[(Tf)(x)=\int_{M_1} K(x,y)f(y)dy,\]
is in the Schatten-von Neumann  classes $S_r(L^2(M_1), L^2(M_2))$ for 
\begin{equation}\label{EQ:mfdr}
\frac 1r<\half +\frac{\mu_1}{n_1}+\frac{\mu_2}{n_2}.
\end{equation} 
Moreover,  its singular numbers satisfy
\begin{equation}\label{EQ:ssmfd}
s_j(T)=o(j^{-\left(\frac12+\frac{\mu_1}{n_1}+\frac{\mu_2}{n_2}\right)}).
\end{equation} 
In particular, for $M=M_1=M_2$, $n=n_1=n_2$:
\begin{itemize}
\item[(i)] If $K\in L^2(M\times M)$ is such that $K\in H^{\mu}(M\times M)$ for 
 $\mu>\frac{n}{2}$, then $T$ is trace class on $L^{2}(M)$ and its trace is given by \eqref{EQ:trace}.
\item[(ii)] If $K\in C_x^{\ell_1} C_{y}^{\ell_2}(M\times M)$ for some even integers $\ell_1,\ell_2\in 2\mathbb N_0$
 such that $\ell_1+\ell_2>\frac n2$, then
 $T$ is trace class on $L^{2}(M)$ and its trace is given by
\begin{equation}\label{EQ:trace2rr}
\Tr(T)=\int_M K(x,x)dx.
\end{equation}
\end{itemize} 
\end{cor}
\begin{proof} In order to prove that $T$ belongs to $S_r(L^2(M_1), L^2(M_2))$ with $r$ satisfying \eqref{EQ:mfdr} we first  recall the following fact: if 
$P\in \Psi^{\nu}_{+e}(M)$ is a positive elliptic pseudo-differential operator of order $\nu>0$ on a closed manifold $M$ of dimension $n$ and $0<p<\infty$ then 
\begin{equation}\label{EQ:mansch}
(I+P)^{-\alpha}\in S_p(L^2(M)) \quad\textrm{ if and only if }\quad \alpha>\frac{n}{p\nu},
\end{equation} 
see \cite[Proposition 3.3]{dr:suffkernel}.
Consequently, condition \eqref{EQ:mfdr} follows from Theorem \ref{ext322} with $E_j=(I+P_j)^{\frac{\mu_j}{\nu_j}}$ for any $P_j\in \Psi^{\nu_j}_{+e}(M_j), \,\,(j=1,2)$. Indeed, since $(I+P_1)_y^{\frac{\mu_1}{\nu_1}}\in S_{p_1}$ for $p_1>\frac{n_1}{\mu_1}$ and  $(I+P_2)_x^{\frac{\mu_2}{\nu_2}}\in S_{p_2}$ for $p_2>\frac{n_2}{\mu_2}$, we have that  $T$ belongs to $S_r(L^2(M_1), L^2(M_2))$ for $r>0$
as in \eqref{EQ:mfdr}.  The rate of decay \eqref{EQ:ssmfd} is now a consequence of Theorem \ref{extahk} and the spectral asymptotics for elliptic pseudo-differential operators on compact manifolds. Furthermore, Part (i) is obtained by letting $r=1$ in  \eqref{EQ:mfdr}, and Part  (ii) follows from Part (i) and formula \eqref{EQ:tracec}.
\end{proof}

This corollary refines the results by the authors in \cite{dr:suffkernel} where the statement \eqref{EQ:mfdr} was obtained in the case $M_1=M_2$. Now this and the refinement of the decay rate in \eqref{EQ:ssmfd} have been  obtained as corollaries of Theorem \ref{ext322} and Theorem \ref{extahk}.

\begin{rem}\label{remsgeom4} (a) We can note that the index $\frac n2$ in Part (ii) of Corollary \ref{ext322l} is in general sharp. For example, for $M=\tn$ being the torus of even dimension $n$, there exist a function $\chi$ of class $C^{\frac n2}$ such that the series of its Fourier coefficients diverges (see \cite[Ch. VII]{ste-we:fa} or \cite{wai:trig}). By considering the convolution kernel $K(x,y)=\chi(x-y)$, the singular values of the operator $T$ given by $Tf=f*\chi$ agree with the absolute values of the Fourier coefficients of $\chi$. Hence, $T\notin S_1(L^2(\tn))$ but $K\in C^{\frac n2}(M\times M)$.
Thus, we see that Part (ii) of Corollary \ref{ext322l} with $\ell_1=0$ and $\ell_2=\frac n2$ is sharp.

This, in turn, justifies the sharpness, in general, for all the results in this paper. \\

(b) An example that arises in spectral geometry is given by the two dimensional double layer potential. Let $\Omega$ be a $C^k$  bounded region in $\er^2$ with $k\geq 2$. Let $E(x,y)=\frac{1}{\pi}\log\frac{1}{|x-y|}$, the  {\em double layer potential} $K:L^2(\partial\Omega)\rightarrow L^2(\partial\Omega)$ is defined as the operator 
\[Kf(x)=\int_{\partial\Omega}\partial\nu_yE(x,y)f(y)dS(y),\]
where $\partial\nu_y$ denotes is the outer normal derivative. The kernel $K(x,y)$ is continuous on $\partial\Omega\times\partial\Omega$, by studying its regularity depending on $k$ as has been applied in \cite{sm:dlay} and using the results for closed manifolds, in this case for $\partial\Omega$ one can determine the rate of decay for the eigenvalues  of the double layer potential from the corresponding membership of the double layer potential to a Schatten-von Neumann class. In particular, one can also deduce trace class properties. We refer to \cite{sm:dlay} for the details on this important example.
\end{rem}

\subsection{Operators on domains with finite measure}
\label{SEC:Riesz}

In Section \ref{SEC:appcm} we considered the case of compact domains. We now discuss the situation when the domains may be unbounded but still have finite measure.

Let $\Omega\subset\Rn$ be a measurable set with finite non-zero Lebesgue measure.
Let us define
$$
\varepsilon_{\alpha,n}(z):=c_{\alpha,n} |z|^{\alpha-n},
$$
with $c_{\alpha,n}=2^{\alpha-n}\pi^{-n/2}\frac{\Gamma(\alpha/2)}{\Gamma((n-\alpha)/2)}.$
Then for $0<\alpha<n$ and $x\in\Omega$ the Riesz potential operator is defined by
\begin{equation}\label{EQ:Riesz}
(\mathcal{R}_{\alpha,\Omega}f)(x):=\int_\Omega \varepsilon_{\alpha,n}(x-y) f(y) dy.
\end{equation}  
Such operators arise naturally as Green functions for boundary value problems for fractional Laplacians on $\Rn$ in view of the relations
\begin{equation}\label{EQ:Rfs}
(-\Delta_y)^{\alpha/2}\varepsilon_{\alpha,n}(x-y) =\delta_x.
\end{equation}  
It was shown in \cite[Proposition 2.1]{RRS-JFA} that the operator $\mathcal{R}_{\alpha,\Omega}$ is non-negative, that is, all of its eigenvalues are non-negative, and satisfies the estimate
\begin{equation}\label{EQ:RRSev}
 \lambda_k(\mathcal{R}_{\alpha,\Omega})=s_k(\mathcal{R}_{\alpha,\Omega})\leq 
 C|\Omega|^{\frac{\alpha}{n}} k^{-\frac{\alpha}{n}}.
\end{equation} 
Indeed, once one shows that the operator $\mathcal{R}_{\alpha,\Omega}$ is non-negative, the estimate \eqref{EQ:RRSev} follows by applying an estimate of 
Cwikel \cite{Cwikel} to the `square root' of the operator $\mathcal{R}_{\alpha,\Omega}$.
The constant $C=C(\alpha,n)$ in \eqref{EQ:RRSev} depends only on $\alpha$ and $n$, and its value can be calculated explicitly, see \cite[Remark 2.2]{RRS-JFA}.
If $\Omega$ is bounded such results go back to Birman and Solomyak \cite{BS-1970}.

As a consequence of \eqref{EQ:RRSev} one readily sees that
operators $\mathcal{R}_{\alpha,\Omega}$ are compact and satisfy
\begin{equation}\label{EQ:Rsch}
\mathcal{R}_{\alpha,\Omega}\in S_p(L^2(\Omega)) \quad\textrm{ for }\; p>\frac{n}{\alpha}.
\end{equation}
Isoperimetric inequalities for operators $\mathcal{R}_{\alpha,\Omega}$ from the point of view of the dependence on $\Omega$ were investigated in \cite{RRS-JFA}.

In view of the relation \eqref{EQ:Rfs} we can write $(-\Delta_\Omega)^{\alpha/2}:=\mathcal{R}_{\alpha,\Omega}^{-1}.$

Applying Theorem \ref{ext322} with Riesz potential operators, we obtain the analogue of Corollary \ref{ext322l} in domains in $\Rn$ with boundaries. 

\begin{cor}\label{COR:Riesz}
Let $\Omega_i\subset{\mathbb R}^{n_i}$, $i=(1,2)$, be measurable sets with finite non-zero Lebesgue measure and let $0<\alpha_i<n_i$. Let 
$K\in L^2(\Omega_2\times \Omega_1)$ be such that we have
$(-\Delta_{\Omega_2})^{\alpha_2/2} (-\Delta_{\Omega_1})^{\alpha_1/2}K\in L^2(\Omega_2\times \Omega_1)$. Then
the integral operator $T$ from $L^2(\Omega_1)$ to $L^2(\Omega_2)$ defined by
 \[(Tf)(x)=\int_{\Omega_1} K(x,y)f(y)dy,\]
is in the Schatten classes $S_r(L^2(\Omega_1), L^2(\Omega_2))$ for 
\begin{equation}\label{EQ:mfdrxx}
\frac 1r<\half +\frac{\alpha_1}{n_1}+\frac{\alpha_2}{n_2}.
\end{equation} 
Moreover,
\begin{multline}\label{par1z}\|T\|_{S_r}\leq \|(-\Delta_{\Omega_1})^{-\frac{\alpha_1}{2}}\|_{S_{\frac{n_1}{\alpha_1}}}\|(-\Delta_{\Omega_2})^{-\frac{\alpha_2}{2}}\|_{S_{\frac{n_2}{\alpha_2}}}
\times \\ \times \|(-\Delta_{\Omega_2})^{\alpha_2/2} (-\Delta_{\Omega_1})^{\alpha_1/2}K\|_{L^2(\Omega_2\times\Omega_1)}.
\end{multline} 

In particular, for $\Omega=\Omega_1=\Omega_2$, $n=n_1=n_2$:
\begin{itemize}
\item[(i)] If $K\in L^2(\Omega\times \Omega)$ is such that $K\in H^{\alpha}(\Omega\times \Omega)$ for 
 $\alpha>\frac{n}{2}$, then $T$ is trace class on $L^{2}(\Omega)$ and its trace is given by \eqref{EQ:trace}.
\item[(ii)] If $K\in C_x^{\ell_1} C_{y}^{\ell_2}(\Omega\times \Omega)$ for some even integers $\ell_1,\ell_2\in 2\mathbb N_0$
 such that $\ell_1+\ell_2>\frac n2$, then
 $T$ is trace class on $L^{2}(\Omega)$ and its trace is given by
\begin{equation}\label{EQ:trace2ty}
\Tr(T)=\int_\Omega K(x,x)dx.
\end{equation}
\end{itemize} 
\end{cor} 

We note that Corollary \ref{COR:Riesz} applies to domains that do not have to be bounded but have finite measure. If the measure of the domain is infinite the regularity of the kernel may not be enough and should be complemented by decay conditions at infinity. Such situations will be considered in Section \ref{anharmonic1} and Section \ref{anharmonic2}.

Applications of Theorem \ref{extahk} to obtain further refinements on the decay rate of singular numbers of $T$ are possible, however, the spectral asymptotics required for its use for fractional Laplacians in $\Omega$ could in general depend on properties of the boundary and boundary conditions, and thus would require further assumptions.

\subsection{Operators on lattices}

In this section we consider operators acting on functions on the integer lattice $\Zn$.
Compared to Section \ref{SEC:appcm}, here the decay conditions at infinity are important while the regularity of the kernel becomes irrelevant (as regularity of a pointwise defined function on a discrete lattice).

We note that as before, Part (i) of the following statement is a special case of Part (ii) with $q=2$, but in the case of $n=m$, i.e. when the integral operator is acting on the same space. It will be useful to employ the operator 
\begin{equation}\label{EQ:Eks}
E^\alpha f(k):=(1+|k|)^\alpha f(k),\quad k\in\Zn.
\end{equation} 
In cases when  there are several variables, we will also write $E_k^\alpha$ for $E^\alpha$ to emphasise that the operator is acting in the variable $k$.

\begin{cor}\label{COR:Zn}
Let $n,m\in\mathbb{N}$. Let $K:\Zn\times\Zm\to\C$ be a function and
 let $T$ be the operator, bounded from $\ell^2(\Zm)$ to $\ell^2(\Zn)$, defined by
 \[(Tf)(k)=\sum_{l\in\Zm} K(k,l)f(l).\]
 Then we have the following properties.
\begin{itemize}
\item[(i)] Assume that for some $\alpha,\beta\geq 0$  we have 
\begin{equation}\label{EQ:ZnKcond}
\|K\|_{\alpha,\beta}^2:=\sum_{k\in\Zn}\sum_{l\in\Zm} (1+|k|)^{2\alpha} (1+|l|)^{2\beta} |K(k,l)|^2<\infty.
\end{equation} 
Then $T\in S_r(\ell^2(\Zm),\ell^2(\Zn))$ for all $0<r<\infty$ such that
\begin{equation}\label{EQ:Znr}
\frac{1}{r}<\frac12+\frac{\alpha}{n}+\frac{\beta}{m}.
\end{equation} 
Moreover,
\beq\label{par1w}\|T\|_{S_r}\leq \|E^{-\alpha_1}\|_{S_{\frac{n_1}{\alpha_1}}}\|E^{-\alpha_2}\|_{S_{\frac{n_2}{\alpha_2}}}\|K\|_{\alpha,\beta}.\eq

The sequence of singular values $(s_j(T))_j$ satisfies the following
estimate for the rate of decay:
\begin{equation}\label{EQ:Zns}
s_j(T)=o(j^{-\left(\half+\frac{\alpha}{n}+\frac{\beta}{m}\right)}).
\end{equation} 
In particular, for $n=m$, if $\alpha,\beta\geq 0$ are such that 
$$\alpha+\beta>\frac{n}{2},$$
then the operator $T$ is trace class on $\ell^2(\Zn)$ and its trace is given by
\begin{equation}\label{EQ:Zntr}
\Tr(T)=\sum_{k\in\Zn} K(k,k)  =\sum_j \lambda_j,
\end{equation}
where $\lambda_j$ are the eigenvalues of the operator $T$ counted with multiplicities.
\item[(ii)] In the case $m=n$, let $1<q\leq 2$ and $\frac1q+\frac{1}{q'}=1$.
Assume that for some $\alpha,\beta\geq 0$  we have 
\begin{equation}\label{EQ:ZnKcond2}
\sum_{l\in\Zn} (1+|l|)^{\beta q'} \left(\sum_{k\in\Zn}(1+|k|)^{\alpha q}  |K(k,l)|^q\right)^{\frac{q'}{q}}<\infty
\end{equation} 
and 
\begin{equation}\label{EQ:ZnKcond3}
\sum_{k\in\Zn} (1+|k|)^{\alpha q'} \left(\sum_{l\in\Zn}(1+|l|)^{\beta q}  |K(k,l)|^q\right)^{\frac{q'}{q}}<\infty.
\end{equation} 
Then $T\in S_r(\ell^2(\Zn))$ for all $0<r<\infty$ such that
\begin{equation}\label{EQ:Znryy}
\frac{1}{r}<\frac1{q'}+\frac{\alpha+\beta}{n}.
\end{equation} 
Moreover, the sequence of singular values $(s_j(T))_j$ satisfies the following
estimate for the rate of decay:
\begin{equation}\label{EQ:Znsve3}
s_j(T)=o(j^{-\left(\frac{1}{q'}+\frac{\alpha+\beta}{n}\right)}).
\end{equation} 
In particular, if $\alpha,\beta\geq 0$ are such that 
$\alpha+\beta>\frac{n}{q},$ then $T$ is trace class on   $\ell^2(\Zn)$ and its trace is given by \eqref{EQ:Zntr}.
\end{itemize} 
\end{cor} 

\begin{proof}
Part (i). We observe that the assumption \eqref{EQ:ZnKcond} 
of Corollary \ref{COR:Zn} can be formulated as
\begin{equation}\label{EQ:kpf1}
E_k^\alpha E_l^\beta K\in \ell^2(\Zn\times\Zm).
\end{equation} 
In order to apply Theorem \ref{extahk} we first note that the Kronecker's delta $\delta_k$ is an eigenfunction of $E^\alpha$ with the eigenvalue $(1+|k|)^\alpha$. Consequently, for $\alpha>0$ we have
$$
N_{E^\alpha}(\lambda)=\#\{k:\, (1+|k|)^\alpha\leq \lambda\}\lesssim 
\#\{k:\, |k|\leq \lambda^{1/\alpha}\}=\lambda^{n/\alpha},
$$
for the operator $E^\alpha$ acting on $\Zn$.
Consequently, by Theorem \ref{extahk}, Part (i) and condition \eqref{EQ:kpf1} we get that
$T\in S_r(\ell^2(\Zm),\ell^2(\Zn))$ provided that 
$$
\frac1r<\frac12+\frac{1}{m/\beta}+\frac{1}{n/\alpha},
$$
implying \eqref{EQ:Znr} for $\beta,\alpha>0$. Otherwise, \eqref{EQ:Znr} follows by Part (ii) of Theorem \ref{extahk}. The decay rate \eqref{EQ:Zns} is another consequence of Theorem \ref{extahk}.

Finally, the trace class condition follows from this by taking $r=1$, in view of Corollary \ref{ext32cor}  and Remark \ref{REM:Lidski}.

Part (ii) follows by the same argument but employing Corollary \ref{COR:extmain} and Remark \ref{REM:Next} instead.
\end{proof}

We note that integral operators on lattices can be also considered as pseudo-difference operators, which is an analogue of pseudo-differential operators on the lattice $\Zn$. Conditions for the membership of such operators in Schatten-von Neumann classes in terms of their symbols were given in \cite{Ruzh-Zn}.

\subsection{Conditions in terms of anharmonic oscillators}
\label{anharmonic1}
In Corollary \ref{COR:Riesz} we considered the case of domains of $\Rn$ of finite measure. We now discuss the case the whole space $\Rn$ when the regularity of the kernel should be complemented by decay conditions at infinity.

For this, we consider a test with the anharmonic oscillator on $L^2(\Rn)$, i.e., the  operator
\[E_a=-\Delta+|x|^a\]
 on $L^2(\Rn)$ for $a>0$. For $a=2$, the operator $E_2=-\Delta+|x|^2$ is the usual harmonic oscillator. 

The study of the harmonic oscillator has been a very active field of research. For the spectral theory of non-commutative versions of the harmonic oscillator we refer to the interesting work of Parmeggiani et al \cite{pab2:b}, \cite{par:1a}, \cite{par:1b}, \cite{par:1c}, \cite{par:1d}, \cite{par:1e} and \cite{par:1f}.\\

Since for operators on $\Rn$ both the regularity and decay of the kernel at infinity are relevant it is natural to try to measure these properties of the kernel by the action of harmonic or anharmonic oscillators. 
The tests with anharmonic oscillators appear to be more natural  when compared to the harmonic oscillator, since the orders of  regularity and decay do not have to be the same.
Moreover, it is natural to consider their fractional powers since the regularity or decay orders do not have to be integers.

Thus, as a consequence of the results of this paper we get the following conditions. 
 
\begin{cor} \label{extah} 
Let $E_a=-\Delta+|x|^a$ on $\Rn$, $E_b=-\Delta+|x|^b$ on $\Rm$ with $a,b>0$.  Let 
$K\in L^2(\Rm\times\Rn)$ and let $T$ be the integral operator  from  
 $L^2(\Rn)$ to $L^2(\Rm)$ defined by
 \[(Tf)(x)=\int_{\Rn} K(x,y)f(y)dy.\]

\begin{enumerate}
\item[(i)] Let $\alpha, \beta\geq 0$. If $(E_b)_x^{\beta}(E_a)_y^{\alpha}K\in L^2(\Rm\times\Rn)$, then $T$ belongs to the Schatten-von Neumann class $S_r(L^2(\Rn),L^2(\Rm))$ for all $0<r<\infty$ such that
\[\frac 1r<\half+\frac {\alpha}{p_a}+\frac {\beta}{p_b},\]
where $p_a=n(\frac{1}{a}+\half)$ and $p_b=m(\frac{1}{b}+\half)$.
Moreover,
\beq\label{par1ww}\|T\|_{S_r}\leq \|E_a^{-1}\|_{S_{\frac{p_a}{\alpha}}}\|E_b^{-1}\|_{S_{\frac{p_b}{\beta}}}\|(E_b)_x^{\beta}(E_a)_y^{\alpha}K\|_{L^2(\Rm\times\Rn)}.\eq

The sequence of singular values $(s_k(T))_k$ satisfies the following
estimate for the rate of decay:
\[s_k(T)=o(k^{-(\half+\frac {\alpha}{p_a}+\frac {\beta}{p_b})}).\]

\item[(ii)] Let $\alpha, \beta\geq 0$. If $m=n$,  $(E_b)_x^{\beta}(E_a)_y^{\alpha}K\in L^2(\Rn\times\Rn)$, and $\half<\frac{\alpha}{p_a}+\frac{\beta}{p_b}$, then
$T$ belongs to the trace class $S_1(L^2(\Rn))$ and its trace is given by
\begin{equation}\label{EQ:trace5}
\Tr(T)=\int_{\Rn}\widetilde{K}(x,x)dx.
\end{equation}
\end{enumerate}
\end{cor}

\begin{proof} The distribution of eigenvalues of $E_a$ and other second order differential operators has been investigated by E. C. Titchmarsh in \cite{titch2:ei}. In particular,
Titchmarsh considered operators of the form $-\Delta+V(x)$ with 
$V(x)\rightarrow\infty$ as $|x|\rightarrow\infty$ and $V(x)$ ultimately non-decreasing on every straight line radiating from the origin. If $N(\lambda)$ denotes the number of eigenvalues less than $\lambda$, then he showed in 
\cite[Section 17.8]{titch2:ei}
that
\[N(\lambda)\sim\frac{1}{2^n\pi^{\frac{n}{2}}\Gamma(\frac{n}{2}+1)}\int_{V<\lambda}\{\lambda-V(x)\}^{\frac{n}{2}}dx, \mbox{ as }\lambda\rightarrow\infty .\]
In particular if $V(x)=|x|^a$ we have 
\[\int_{|x|^a<\lambda}\{\lambda-|x|^a\}^{\frac{n}{2}}dx=C\int\limits_0^{\lambda^{\frac 1a}}\left(\lambda-r^a\right)^{\frac{n}{2}}r^{n-1}dr\leq C\int\limits_0^{\lambda^{\frac 1a}}r^{\frac{na}{2}+n-1}dr.\]
Since 
\[\int\limits_0^{\lambda^{\frac 1a}}r^{\frac{na}{2}+n-1}dr=\frac{\lambda^{n(\frac{1}{a}+\half)}}{n(\frac{1}{a}+\half)}.\]

We obtain 
\[N(\lambda)\sim C\lambda^{p_a} \mbox{ as }\lambda\rightarrow\infty ,\]
where $p_a=n(\frac{1}{a}+\half)$.\\

Now, since $\lambda$ is an eigenvalue of $E_a$ if and only if $\lambda^{\alpha}$ is an eigenvalue of $(E_a)^{\alpha}$, we obtain
\beq N_{(E_a)^{\alpha}}(\lambda)=N_{E_a}(\lambda^{\frac 1{\alpha}})\leq C\lambda^{\frac{p_a}{\alpha}}\label{njgtk},\eq
where $N_{P}$ denotes the counting eigenvalue function for the operator $P$.
Then we have
\[
\sum\limits_{k=1}^{\infty}\lambda_k^{-\alpha q}<\infty, \,\,\mbox{ for all }\,\,q>\frac{p_a}{\alpha}.
\]
The singular values of $(E_a^{\alpha})^{-1}$ are $s_k((E_a^{\alpha})^{-1})=\lambda_k^{-\alpha}$ and $(E_a^{\alpha})^{-1}\in S_q(L^2(\Rn))\,\mbox{ for all }\,\,q>\frac{p_a}{\alpha}$. In a similar way we also have  $(E_b^{\beta})^{-1}\in S_{q'}(L^2(\Rm))\,\mbox{ for all }\,\,q'>\frac{p_b}{\beta}$. 

As a consequence of Theorem \ref{ext322}  we obtain
\[T\in S_r(L^2(\Rn),L^2(\Rm)) \]
for 
\[\frac 1r<\half+\frac {\alpha}{p_a}+\frac {\beta}{p_b}.\]
The rate of decay and Part (ii) now  follow from Theorem \ref{extahk}. 
\end{proof}

\begin{ex}
Let us give a simple example for Corollary \ref{extah}: let $T:L^2(\Rn)\to L^2(\Rn)$ be an integral operator with kernel $K(x,y)$. Assume that $1\leq n\leq 3$ and that
$\Delta_x K, (1+|x|^b) K \in L^2(\Rn\times\Rn)$. Then $T$ is a trace class operator provided that $b>\frac{2n}{4-n}.$
\end{ex}
This statement follows immediately from Part (ii) of Corollary \ref{extah} by taking 
$\alpha=0$, $\beta=1$, implying that $T$ is trace class provided that $p_b=n(\frac{1}{b}+\half)<2$.

\subsection{Higher order anharmonic oscillators}
\label{anharmonic2}

One can also get a number of similar tests based on the estimation of $N(\lambda)$ for different operators and the arguments in Corollary \ref{extah}. Here we will
consider different examples of anharmonic oscillators however restricting to integer orders of derivatives and weights.

More specifically, let us consider the operator $$E= (-\Delta)^k+|x|^{2\ell}$$ on $\Rn$, $n\geq 1$, where $k,\ell$ are integers $\geq 1$.

It is well known that such $E$ has a discrete spectrum (see \cite{{shubin:r}}) and it was also shown in \cite[Theorem 3.2]{BBR-book} that for large $\lambda$ the eigenvalue counting function $N(\lambda)$ is bounded by $C\int_{a(x,\xi)<\lambda} dx d\xi$, where $a(x,\xi)$ is the Weyl symbol of the partial differential operator $E$.
 By the change of variables $\xi=\lambda^{1/2k}\xi'$ and $x=\lambda^{1/2\ell}x'$, we can estimate
for large $\lambda$ that
 \begin{equation}\label{EQ:anhnl}
N(\lambda)\lesssim \iint_{|\xi|^{2k}+|x|^{2\ell}<\lambda} dx d\xi= \lambda^{n(\frac{1}{2k}+\frac{1}{2\ell})}
  \iint_{|\xi'|^{2k}+|x'|^{2\ell}<1} dx' d\xi' \lesssim \lambda^{n(\frac{1}{2k}+\frac{1}{2\ell})}.
\end{equation} 
 We note that refined estimates for the remainder in the spectral asymptotics for $N(\lambda)$ 
 were also studied by Helffer and Robert in   \cite[Theorem 6 and Corollary 2.7]{hr:anosc2} in the case $k=\ell$, and 
 in \cite{hr:anosc4} for different $k$ and $\ell$  in the case $n=1$. 
 
 Moreover, all the results remain unchanged if we add lower order terms to the operator $E$.
 
 Consequently, from Theorem \ref{extahk} and arguing similarly as in the  proof of Corollary \ref{extah} we obtain:

\begin{cor} \label{extahm1} 
Let $E_i= (-\Delta)^{k_i}+|x|^{2\ell_i}$ be operators on $\er^{n_i}$, where
$n_i, k_i, \ell_i$ are integers $\geq 1$ for $i=1,2.$
Let us set $p_i:=\frac{n}{2}(\frac 1{k_i}+\frac 1{\ell_i}),\,\, i=1,2.$ 

Let  $K\in L^2(\er^{n_2}\times\er^{n_1})$ and let $T$ be the integral operator  from  
 $L^2(\er^{n_1})$ to $L^2(\er^{n_2})$ defined by
 \[(Tf)(x)=\int_{\er^{n_1}} K(x,y)f(y)dy.\]
Let $\alpha,\beta\geq 0$ and suppose that $(E_2)_x^{\beta}(E_1)_y^{\alpha}K\in L^2(\er^{n_2}\times\er^{n_1})$.   
Then $T$ belongs to the Schatten-von Neumann class $S_r(L^2(\er^{n_1}),L^2(\er^{n_2}))$ for
all $0<r<\infty$ such that
\[\frac 1r<\half+\frac {\alpha}{p_1}+\frac {\beta}{p_2}.\]
Moreover,
\beq\label{par1wew}\|T\|_{S_r}\leq \|E_1^{-\alpha}\|_{S_{\frac{p_1}{\alpha}}}\|E_2^{-\beta}\|_{S_{\frac{p_2}{\beta}}}\|(E_2)_x^{\beta}(E_1)_y^{\alpha}K\|_{L^2(\er^{n_2}\times\er^{n_1})}.\eq

The sequence of singular values $(s_j(T))_j$ satisfies the following
estimate for the rate of decay:
\[s_j(T)=o(j^{-(\half+\frac {\alpha}{p_1}+\frac {\beta}{p_2})}).\]
\end{cor}

\begin{ex}
Let us give a simple example for Corollary \ref{extahm1}: let $T:L^2(\er)\to L^2(\er)$ be an integral operator with kernel $K(x,y)$. Assume that $k,l\in\mathbb{N}$ and that
$K, \frac{d^{2k}}{dx^{2k}} K, x^{2l} K \in L^2(\er\times\er)$. Then $T\in S_r(L^2(\er))$ provided that $\frac{1}{r}<\frac12+\frac{2kl}{k+l}$. In particular, under the above assumptions $T$ is always a trace class operator.
\end{ex}
This statement follows immediately from Corollary \ref{extahm1} by taking 
$\alpha=0$ and $\beta=1$.

\begin{rem} We would like now to consider the special case of a negative order and negative potential, more precisely the case of the hydrogen atom, i.e. an operator of the form $H= -\Delta-c|x|^{-1}$ on $\er^3$ with  $c>0$. It is well known that the energy levels  are of the form 
\[E_n=-\frac{C}{n^2}\] 
where $C$ is a positive constant. In this case we can take $E^{-1}=H$ which belongs to the class $S_p$ with $p>\half$. Therefore, one can obtain a similar result to Corollary  \ref{extah} in terms of the operator $E$ with an index $p>\half$.
\end{rem}

\subsection{Subelliptic conditions on sub-Riemannian manifolds}
\label{SEC:subR}

In general, once the upper bound for the eigenvalue counting function of a certain operator is 
obtained, it can be used in Theorem \ref{extahk}. In particular, in some situations is may be convenient to use operators respecting certain geometric structures. Rather general results on the spectral asymptotics for self-adjoint subelliptic operators have been obtained by 
Fefferman and Phong \cite{FP-PNAS,FP-subelliptic-1983} as well as for 
operators with double characteristics by Menikoff and Sj\"ostrand \cite{MS-MA}, see also an overview on spectral asymptotics for rather general hypoelliptic operators by Sj\"ostrand \cite{Sj-ICM} and more recent extensions by Ponge \cite{Ponge-MAMS} and Hassannezhad and Kokarev 
\cite{HK2015}.
We can also refer to \cite{FR16} for subelliptic analysis on nilpotent groups and to \cite{RS-AM} for the potential theory for the sub-Laplacians.

\smallskip
Let us formulate several examples but  first we briefly recall a few definitions.
Let $M$ be a connected closed manifold and let $H\subset TM$ be a smooth sub-bundle of the tangent bundle satisfying the H\"ormander condition. we recall that the sub-bundle $H$ satisfies the  {\em H\"ormander condition} if for any point $x\in M$ and any local frame $\{X_i\}$ of $H$ around $x$, the iterated Lie brackets $[X_i,X_j]$, $[[X_i,X_j], X_k]$, 
$[X_i,[...[X_j, X_k]...]]$ at $x$ together with the vectors $\{X_i(x)\}$ span the  tangent space $T_xM$. The length of the Lie bracket above is understood as the number of vector fields involved. 
The sub-bundle $H$ is called regular if the dimensions of the strata in the stratification of $T_x M$ by commutators do not depend on $x\in M$. 
Let $g$ be a smooth metric on $H$ and let $Q$ be the Hausdorff dimension of $M$ with respect to the Carnot-Caratheodory distance associated to the sub-Riemannian manifold $(M,H,g)$.

We recall that the eigenvalue counting function of the sub-Laplacian on compact regular sub-Riemannian manifolds is estimated by $N(\lambda)\leq C\lambda^{Q/2}$, see e.g. \cite{HK2015}.
Consequently, Theorem \ref{extahk} immediately implies:

\begin{cor}\label{COR:subR}
Let $(M_i,H_i,g_i)$ $(i=1,2)$ be compact regular sub-Riemannian manifolds and let $\Delta_i$ be the sub-Laplacians associated to $H_i$. Let $Q_i$ denote the Hausdorff dimensions of $M_i$ with respect to the respective Carnot-Caratheodory distances.

Let 
$K\in L^2(M_2\times M_1)$ and let $T$ be the integral operator  from  
 $L^2(M_1)$ to $L^2(M_2)$ defined by
 \[(Tf)(x)=\int_{M_1} K(x,y)f(y)dy.\]
 Let $\alpha, \beta\geq 0$ and assume that $(\Delta_2)_x^{\beta}(\Delta_1)_y^{\alpha}K\in L^2(M_2\times M_1)$.
Then $T$ belongs to the Schatten-von Neumann class $S_r(L^2(M_1),L^2(M_2))$ for
all $0<r<\infty$ such that
\[\frac 1r<\half+\frac {2\alpha}{Q_1}+\frac {2\beta}{Q_2},\]
Moreover
\beq\label{par1wewq}\|T\|_{S_r}\leq \|\Delta_1^{-\alpha}\|_{S_{\frac{Q_1}{2\alpha}}}\|\Delta_2^{-\beta}\|_{S_{\frac{Q_2}{2\beta}}}\|(\Delta_2)_x^{\beta}(\Delta_1)_y^{\alpha}K\|_{ L^2(M_2\times M_1)
}.\eq

The sequence of singular values $(s_j(T))_j$ satisfies the following
estimate for the rate of decay:
\[s_j(T)=o(j^{-(\half+\frac {2\alpha}{Q_1}+\frac {2\beta}{Q_2})}).\]
\end{cor} 

Let us briefly record two examples that are of particular importance: of compact contact manifolds and of compact Lie groups.

We recall that a contact manifold is a smooth manifold $M$ of odd dimension $2n+1$ equipped with an 1-form $\theta$ such that $\theta\wedge (d\theta)^n$ is a volume form on $M$. The canonically induced bundle 
$H_x:=\{X\in T_x M: \theta(X)=0\}$ is regular and satisfies H\"ormander's condition since 2-form $d\theta$ is non-degenerate on $H$. This will be the setting (C1) in the following statement. The setting (C2) concerns sub-Laplacians on compact Lie groups in which case the canonical sub-bundle is also regular due to the left-invariance.

\begin{cor}\label{COR:subexamples}
Let us consider the following situations:
\begin{enumerate}
\item[(C1)] Let $\Omega_i$ be a compact contact metric manifold of dimension $2n_i+1$, $(i=1,2)$.
Let $E_i:=(I+\Delta_i)^{\alpha_i}$ for $i=1,2$, where $\Delta_i$ is the canonical positive sub-Laplacian on $\Omega_i$.  Let $p_i:=n_i+1,\,\, i=1,2.$ 

\item[(C2)] Let $\Omega_i$ be a compact Lie group with left-invariant positive sub-Laplacian $\mathcal{L}_i$, and let $Q_i$ be the Hausdorff dimension of the induced control distance. 
Let $E_i:=(I+\mathcal{L}_i)^{\alpha_i}$ and let $p_i:=\frac{Q_i}{2},\,\, i=1,2.$ 
\end{enumerate}
Let  $K\in L^2(\Omega_2\times\Omega_1)$ and let $T$ be the integral operator  from  
 $L^2(\Omega_1)$ to $L^2(\Omega_2)$ defined by
 \[(Tf)(x)=\int_{\Omega_1} K(x,y)f(y)d\mu_1(y).\]
Let $\alpha,\beta\geq 0$ and suppose that $(E_2)_x^{\beta}(E_1)_y^{\alpha}K\in L^2(\Omega_2\times\Omega_1)$ under the corresponding assumptions either {\rm (C1)} or {\rm (C2)}.   
Then $T$ belongs to the Schatten-von Neumann class $S_r(L^2(\Omega_1),L^2(\Omega_2))$ for
all $0<r<\infty$ such that
\[\frac 1r<\half+\frac {\alpha_1}{p_1}+\frac {\alpha_2}{p_2}.\]

Moreover, the sequence of singular values $(s_j(T))_j$ satisfies the following
estimate for the rate of decay:
\[s_j(T)=o(j^{-(\half+\frac {\alpha_1}{p_1}+\frac {\alpha_2}{p_2})}).\]
\end{cor} 
Again, Corollary \ref{COR:subexamples} is an immediate consequence of Theorem \ref{extahk} and the corresponding spectral asymptotics results.

Let us mention two further important special cases of the settings (C1) and (C2) of 
Corollary \ref{COR:subexamples}:

\smallskip 
(C1) In particular, the result of the setting (C1) also holds with the same indices if any of the manifolds $M_i$ is a connected orientable compact strictly pseudo-convex CR manifold of dimension $2n+1$. For the required  spectral asymptotics see, e.g. \cite{Kokarev-CR}. In this case one considers  asymptotics for the counting function of a sub-Laplacian corresponding to a pseudo-Hermitian structure.\smallskip

(C2) If $\mathcal{L}_i$ is not the sub-Laplacian but a Laplacian (Casimir element) on a compact Lie group $\Omega_i$ of dimension $n_i$ then we have $Q_i=n_i$.
In the setting of operators on  a compact Lie group $G$ conditions for the membership in Schatten-von Neumann classes were given in \cite{dr13:schatten}
also in terms of global matrix symbols on $G\times\widehat{G}$. We can refer to \cite{rt:book, rt:groups} for the corresponding analysis and its relations to the representation theory of compact Lie groups. 

\smallskip
In analogy to \eqref{EQ:mixed-Sobolev}, the above conditions on the kernel can be also formulated in terms of the (mixed) Sobolev spaces associated to the sub-Laplacians.
The embeddings between these Sobolev spaces and the usual ones can be obtained from   a suitable $S(m,g)$ calculus when available. However, we can note that already for the Sobolev spaces associated to harmonic oscillators, those Sobolev spaces take into account also decay properties at infinity, while the usual Sobolev spaces do not. So, one can compare these spaces locally (but there terms like $|x|^2$ do not play any role), but globally there may be embeddings only in one directions.\\

\medskip

\noindent{\bf Acknowledgments}

The authors were supported by the Leverhulme Research Grant RPG-2017-151, by the 
 FWO Odysseus Project G.0H94.18N: Analysis and Partial Differential Equations, and by the EPSRC grant EP/R003025. The first author was also supported by Vic. Inv Universidad del Valle. Grant No. CI-71281.  The authors would also like to thank anonymous referees for the valuable comments helping to improve the results, the presentation of the manuscript and in particular to simplify the proof of Theorem \ref{ext322}.



\bibliographystyle{alphaabbr}

\end{document}